\documentclass[12pt]{amsart}

\usepackage{amsthm, amsmath, amsfonts, amssymb, graphicx, tikz-cd, comment, array, hyperref, extarrows, mathtools, todonotes}
\usepackage{enumitem}
\usepackage{breakcites} 
\usepackage{calrsfs} 
\usepackage{stmaryrd} 
\DeclareMathAlphabet{\pazocal}{OMS}{zplm}{m}{n} 
\DeclareSymbolFont{yhlargesymbols}{OMX}{yhex}{m}{n} 
\DeclareMathAccent{\yhwidehat}{\mathord}{yhlargesymbols}{"62}

\newtheorem{theorem}{Theorem}[section]

\newtheorem{proposition}[theorem]{Proposition}
\newtheorem{corollary}[theorem]{Corollary}

\theoremstyle{definition}
\newtheorem{definition}[theorem]{Definition}

\newtheorem{remark}[theorem]{Remark}

\setlist[enumerate,1]{label={\upshape(\arabic*)},ref=\arabic*} 

\def\KHaus{{\mathbf{KHaus}}}
\def\Stone{{\mathbf{Stone}}}
\def\DLat{\mathbf{BDL}}
\def\Heyt{\mathbf{HA}}
\def\Frm{\mathbf{Frm}}

\newcommand{\BA}{\mathbf{BA}}
\newcommand{\ED}{{\mathbf{ED}}}
\newcommand{\cBA}{{\mathbf{cBA}}}

\newcommand{\MT}{{\mathbf{MT}}}
\newcommand{\CBA}{{\mathbf{CBA}}}
\newcommand{\CHA}{{\mathbf{CHA}}}
\newcommand{\Set}{{\mathbf{Set}}}
\newcommand{\C}{\mathbf{C}}
\newcommand{\E}{\mathbf{E}}
\newcommand{\LPries}{\mathbf{LSp_{st}}}
\newcommand{\FrmBL}{\Frm_{\mathbf{BL}}}
\newcommand{\HeytBL}{\Heyt_{\mathbf{BL}}}
\newcommand{\FrmHA}{\Frm_{\mathbf{HA}}}

\newcommand{\KHausR}{\mathbf{KHausR}}
\newcommand{\Pries}{\mathbf{Pries}}

\newcommand{\StoneC}{\mathbf{StoneCR}}
\newcommand{\Esast}{\mathbf{Esa_{st}}}
\newcommand{\KHausT}{\mathbf{KHausR_T}}
\newcommand{\KHausB}{\mathbf{KHausR_B}}

\newcommand{\KHausTB}{\mathbf{KHausR_{TB}}}

\newcommand{\StoneCT}{\mathbf{StoneC_T}}
\newcommand{\StoneCKfour}{\mathbf{StoneC_{K4}}}
\newcommand{\StoneCSfour}{\mathbf{StoneCQ}}
\newcommand{\StoneCTB}{\mathbf{StoneC_{TB}}}
\newcommand{\StoneCB}{\mathbf{StoneC_{B}}}
\newcommand{\StoneCSfive}{\mathbf{StoneC_{S5}}}

\newcommand{\CBAO}{\mathbf{CBAO}_{\mathbf{st}}}
\newcommand{\cBAO}{\mathbf{cBAO}_{\mathbf{st}}}
\newcommand{\BAO}{\mathbf{BAO}_{\mathbf{st}}}
\newcommand{\CA}{\mathbf{CA}_{\mathbf{st}}}
\newcommand{\Kfour}{\mathbf{K4}_{\mathbf{st}}}
\newcommand{\MA}{\mathbf{S5}_{\mathbf{st}}}
\newcommand{\B}{\mathbf{B}_{\mathbf{st}}}
\newcommand{\T}{\mathbf{T}_{\mathbf{st}}}
\newcommand{\TB}{\mathbf{TB}_{\mathbf{st}}}
\newcommand{\CKfour}{\mathbf{CK4}_{\mathbf{st}}}
\newcommand{\CMA}{\mathbf{CS5}_{\mathbf{st}}}
\newcommand{\CB}{\mathbf{CB}_{\mathbf{st}}}
\newcommand{\CT}{\mathbf{CT}_{\mathbf{st}}}
\newcommand{\CTB}{\mathbf{CTB}_{\mathbf{st}}}
\newcommand{\cKfour}{\mathbf{cK4}_{\mathbf{st}}}
\newcommand{\cMA}{\mathbf{cS5}_{\mathbf{st}}}
\newcommand{\cB}{\mathbf{cB}_{\mathbf{st}}}
\newcommand{\cT}{\mathbf{cT}_{\mathbf{st}}}
\newcommand{\MTst}{\mathbf{MT}_{\mathbf{st}}}
\newcommand{\cTB}{\mathbf{cTB}_{\mathbf{st}}}

\newcommand{\N}{\mathbb{N}}

\newcommand{\tworef}{\mathbf{2}}
\newcommand{\twoirr}{
{
\bf 2_{\tikz[baseline=-0.5ex]{
\node[circle,draw,inner sep=2pt,fill=black] (a) at (0,0) {};
}}
}
}
\newcommand{\tworefsym}{
{
\bf 2^d_{\tikz[baseline=-0.5ex]{
\node[circle,draw,inner sep=2pt,fill=white] (a) at (0,0) {};
}}
}
}
\newcommand{\twoirrsym}{
{
\bf 2^d_{\tikz[baseline=-0.5ex]{
\node[circle,draw,inner sep=2pt,fill=black] (a) at (0,0) {};
}}
}
}
\newcommand{\oneirr}{
{
\bf 1_{\tikz[baseline=-0.5ex]{
\node[circle,draw,inner sep=2pt,fill=black] (a) at (0,0) {};
}}
}
}

\pgfdeclarelayer{edgelayer}
\pgfdeclarelayer{nodelayer}
\pgfsetlayers{edgelayer,nodelayer,main}

\tikzstyle{none}=[inner sep=0pt]
\tikzstyle{black dot}=[fill=black, draw=black, shape=circle, minimum size=5pt, inner sep=0]
\tikzstyle{small black dot}=[fill=black, draw=black, shape=circle, minimum size=0.75pt, inner sep=0]
\tikzstyle{white dot}=[fill=white, draw=black, shape=circle, minimum size=7pt, inner sep=0]
\tikzstyle{big black dot}=[fill=black, draw=black, shape=circle, minimum size=7pt, inner sep=0]

\tikzstyle{none dashed}=[dashed, -]
\tikzstyle{to}=[->]
\tikzstyle{doubleto}=[<->]

\setlist[enumerate,1]{label={\upshape(\arabic*)},ref=\arabic*} 

\title{On the lack of colimits in various categories arising in pointfree topology and algebraic logic}

\author{Marco Abbadini}
\address{University of Birmingham, School of Computer Science,
B15 2TT Birmingham (UK)}
\email{m.abbadini@bham.ac.uk}

\author{Guram Bezhanishvili}
\address{Department of Mathematical Sciences\\
New Mexico State University\\
Las Cruces NM 88003\\
USA}
\email{guram@nmsu.edu}

\author{Luca Carai}
\address{Department of Mathematics ``Federigo Enriques''\\
University of Milan\\
20133 Milan\\
Italy}
\email{luca.carai.uni@gmail.com}

\date{}

\begin{document}


\begin{abstract}
    We prove that the category of McKinsey-Tarski algebras is not equivalent to a variety of algebras, thus answering a question of Peter Jipsen in the negative. More generally, we show that various categories of BAOs (boolean algebras with an operator), Heyting algebras, and frames with appropriate morphisms between them are not cocomplete.
    As a consequence, none of these categories is equivalent to a prevariety, let alone a variety.
\end{abstract}
\maketitle

\tableofcontents

\section{Introduction}

Boolean algebras and Heyting algebras are some of the most studied classes of algebras in algebraic logic. It is well known that both classes are closed under homomorphic images, subalgebras, and products, hence form a variety. This allows the use of powerful tools from universal algebra in their study. The situation changes when we restrict our attention to the categories $\CBA$ of complete boolean algebras with complete boolean morphisms and $\CHA$ of complete Heyting algebras with complete Heyting morphisms.
Indeed, it is a classic result from the 1960s that the free complete boolean algebra on countably many generators does not exist \cite{Gai64,Hal64}. De Jongh \cite{DeJ80} proved that already the two-generated free complete Heyting algebra does not exist (see also \cite{BK24}). Thus, neither $\CBA$ nor $\CHA$ is equivalent to a variety.
Natural generalizations of varieties are quasi-varieties (classes of algebras closed under isomorphisms, subalgebras, products, and ultraproducts), which further generalize to prevarieties (classes of algebras closed under isomorphisms, subalgebras, and products). 
It is well known that each prevariety viewed as a category is complete and cocomplete (see, e.g., 
\cite[Thm.~IV.2.1.3 and IV.2.2.3]{SR99}). Therefore, each category that is not cocomplete is not equivalent to a prevariety.
Consequently, neither $\CBA$ nor $\CHA$ is equivalent to even a prevariety.

There are other natural morphisms to consider between complete boolean algebras and between complete Heyting algebras. For example, we can look at the category of complete boolean algebras with boolean morphisms between them, as well as the category of complete Heyting algebras with Heyting morphisms, or, more generally, bounded lattice morphisms between them. We show that none of these categories is equivalent to a prevariety, and further generalize this to prove that the category of Heyting algebras with bounded lattice morphisms between them is also not equivalent to a prevariety.

In addition, we look at BAOs (boolean algebras with an operator $\Diamond$) and show that a similar phenomenon occurs there. The morphisms we consider are \emph{stable morphisms} (that is, boolean morphisms $f \colon A\to B$ satisfying $\Diamond f(a) \le f(\Diamond a)$ for each $a\in A$), which play an important role in the study of the finite model property in modal logic \cite{Ghi10,BBI16}. Such morphisms also play an important role in the study of \emph{The Algebra of Topology} of McKinsey and Tarski \cite{MT44,BR23}. Our article is motivated by the question of Peter Jipsen, who asked, at the \href{https://sites.google.com/view/frametheoryseminar/frame-theory-workshop?authuser=0}{Frame Theory Workshop} (Chapman University, December 2024), whether the category $\MT$ of McKinsey-Tarski algebras with complete stable morphisms is (equivalent to) a variety. We answer this question in the negative by showing that $\MT$ is not equivalent even to a prevariety. In fact, we show that neither $\MT$ nor the category of McKinsey-Tarski algebras with stable morphisms
is equivalent to a prevariety. These results generalize to various categories of complete BAOs with these two types of morphisms, as well as to various categories of BAOs with stable morphisms. 

We briefly describe the flavor of our results. In Section~\ref{sec: CBAOs} we show that $\MT$ lacks some countable copowers, and hence is not equivalent to a prevariety. We then generalize this result to various categories of complete BAOs with complete stable morphisms. 
In Section~\ref{sec: cBA} we prove that dropping the completeness assumption from the above notion of morphism results in categories that lack even some binary copowers. 
In Section~\ref{sec: 3.2} we show that dropping the completeness assumption also from the objects results in categories that lack some coequalizers. 
Finally, in Section~\ref{sec: 5} we prove that the category of frames with Heyting morphisms lacks some binary copowers, and so does the category of frames with bounded lattice morphisms. Furthermore, we prove that the category of Heyting algebras with bounded lattice morphisms lacks some coequalizers.

In Sections~\ref{sec: cBA} to \ref{sec: 5} our main tool is duality theory: Stone duality for boolean algebras, J\'onsson-Tarski duality for BAOs, and Esakia and Priestley dualities for Heyting algebras and bounded distributive lattices. Dualizing the fact that prevarieties are cocomplete, it is then sufficient to show that the dual categories in question are not complete. This we do by showing that in these categories either products or equalizers do not exist.

\section{MT-algebras}\label{sec: CBAOs}

In this section, we show that $\MT$ is not cocomplete.
As a consequence, it is not equivalent to a variety, resolving the question of Peter Jipsen in the negative. We also show that the same technique applies to several other categories of complete BAOs with complete stable morphisms, yielding that they are not cocomplete, and hence not equivalent to a prevariety.

We recall (see \cite{JT51}) that an \emph{operator} on a boolean algebra $B$ is a unary function $\Diamond \colon B \to B$ preserving finite joins, and that the pair $(B,\Diamond)$ is called a \emph{boolean algebra with an operator} or simply a \emph{BAO}. 
A \emph{BAO-morphism} between two BAOs $A$ and $B$ is a boolean morphism $f \colon A \to B$ such that $f(\Diamond a)=\Diamond f(a)$ for each $a\in A$. 
We let $\mathbf{BAO}$ denote the category of BAOs and BAO-morphisms.

We will mainly be interested in \emph{stable} morphisms between BAOs \cite{BBI16}; that is, boolean morphisms $f \colon A \to B$ such that $\Diamond f(a) \le f(\Diamond a)$ for each $a\in A$. Stable morphisms are also known as \emph{continuous morphisms} and play an important role in the study of axiomatization, finite model property, and decidability of modal logics \cite{Ghi10,BBI16}. 

An important class of BAOs is formed by closure algebras of McKinsey and Tarski \cite{MT44}. We recall that a \emph{closure algebra} is a BAO $(B,\Diamond)$ satisfying $a \le \Diamond a$ and $\Diamond\Diamond a \le \Diamond a$ for each $a \in B$. Closure algebras play a prominent role in modal logic as they serve as algebraic models of $\mathbf{S4}$, one of the most studied modal systems (see, e.g., \cite{RS70}). Because of this, closure algebras are also known as $\mathbf{S4}$-algebras.
Complete closure algebras provide an alternate pointfree approach to topology, and were coined McKinsey-Tarski algebras in \cite{BR23}.
We thus arrive at the main definition of this section:

\begin{definition}
    A \emph{McKinsey-Tarski algebra}, or simply an \emph{MT-algebra}, is a complete closure algebra. An \emph{MT-morphism} between MT-algebras is a complete boolean morphism that is stable. Let $\MT$ denote the category of MT-algebras and MT-morphisms. 
\end{definition}

Our aim is to show that $\MT$ lacks some countable copowers. For this, we utilize the following well-known result from the 1960s. Let $\mathbf{BA}$ denote the category of boolean algebras and boolean morphisms, and $\CBA$ the category of complete boolean algebras and complete boolean morphisms.

\begin{theorem}[\cite{Gai64,Hal64}] \label{thm: Gai Hal}
    The free countably generated complete boolean algebra does not exist. 
    Thus, $\CBA$ is not cocomplete, and hence is not equivalent to a prevariety.
\end{theorem}

We clearly have the forgetful functor $\mathcal{U} \colon \MT \to \CBA$. We start by proving that it has both left and right adjoints. Following \cite{Hal56}, for each boolean algebra $B$, we consider the so-called \emph{simple operator} $\Diamond_s \colon B \to B$ given by
\[
    \Diamond_s a=
    \begin{cases}
        0 & \text{if } a = 0\\
        1 & \text{if } a \ne 0
    \end{cases}
\]
for each $a \in B$.
It is straightforward to see that $(B,\Diamond_s)$ is a closure algebra, and hence an MT-algebra whenever $B$ is complete.

\begin{theorem} \label{t:adjoints}
    The forgetful functor $\mathcal{U} \colon \MT \to \CBA$ is both a left and right adjoint, and hence preserves both colimits and limits.
\end{theorem}

\begin{proof}
    We define the functor $\mathcal{L} \colon \CBA \to \MT$ as follows. On objects, $\mathcal{L}(A) = (A, \Diamond_s)$; and on morphisms, $\mathcal{L}$ is the identity.
    We prove that, for $A \in \CBA$ and $B \in \MT$, we have 
    \[
    \hom_{\MT}(\mathcal{L} (A),B) \cong \hom_{\CBA}(A,\mathcal{U}(B)) \]
    where the bijection is the identity.
    The inclusion $\subseteq$ is clear since each $\MT$-morphism is a complete boolean morphism.
    For the reverse inclusion, suppose $f \colon A \to \mathcal{U}(B)$ is a complete boolean morphism. We claim that it is a stable morphism from $\mathcal{L}(A)$ to $B$, i.e.\ that $\Diamond f(a) \leq f(\Diamond_s a)$.
    If $a = 0$, then 
    \[
    \Diamond f(0) = \Diamond 0 = 0 = f(0) = f(\Diamond_s 0).
    \]
    If $a \neq 0$, then 
    \[
    \Diamond f(a) \leq 1 = f(1) = f(\Diamond_s a).
    \]
    Thus, $f$ is an $\MT$-morphism.
    
    We next define the functor $\mathcal{R} \colon \CBA \to \MT$ as follows.
    On objects, $\mathcal{R}(A) = (A, \Diamond_i)$, where $\Diamond_i$ is the identity on $A$; and on morphisms, $\mathcal{R}$ is the identity. Since $(A,\Diamond_i)$ is an MT-algebra, $\mathcal R$ is well defined. 
    We prove that, for $A \in \CBA$ and $B \in \MT$, we have
    \[
        \hom_{\MT}(B, \mathcal{R}(A)) = \hom_{\CBA}(\mathcal{U}(B),A).
    \]
    The inclusion $\subseteq$ is clear because each $\MT$-morphism is a complete boolean morphism.
    For the reverse inclusion, suppose $g \colon \mathcal{U}(B) \to A$ is a $\CBA$-morphism. Since $(B,\Diamond)$ is an MT-algebra, $a \leq \Diamond a$ for every $a \in B$. Therefore, $g$ is a stable morphism from $B$ to $\mathcal{R}(A)$ because, for every $a \in B$, we have $\Diamond_i g(a) = g(a) \leq g(\Diamond a)$. Thus, $g$ is an $\MT$-morphism.
\end{proof}

\begin{theorem} \label{thm: MT not cocomplete}
    Let $\C$ be a category and $\mathcal{U} \colon \C \to \CBA$ a colimit-preserving functor.
    If there is $X \in \C$ such that $\mathcal{U}(X)$ is the four-element boolean algebra, then $\C$ lacks some countable copowers.
\end{theorem}

\begin{proof}
    Suppose all countable copowers exist in $\C$.
    Since $\mathcal{U} \colon \C \to \CBA$ preserves colimits, there exists a countably-indexed copower $C$ of $\mathcal{U}(X)$ in $\CBA$.
    Because $\BA$ and $\CBA$ share finitely generated free objects (see, e.g., \cite[p.~33]{Joh82}), the four-element boolean algebra $\mathcal{U}(X)$ is the free object on one generator in $\CBA$, and so $C$ is the free object on $\omega$ generators in $\CBA$, contradicting Theorem~\ref{thm: Gai Hal}. 
\end{proof}

As an immediate consequence of Theorems~\ref{t:adjoints} and \ref{thm: MT not cocomplete} we obtain: 

\begin{theorem} \label{thm: MT lack copowers}
    $\MT$ lacks some countable copowers, and hence is not equivalent to a prevariety.
\end{theorem}

Apart from closure algebras, there are other classes of BAOs that play a prominent role in modal logic (see, e.g., \cite{CZ97}). We conclude the section by showing how to extend Theorem~\ref{thm: MT lack copowers} to those classes. 

\begin{definition}\ \label{def: algebras}
    \begin{enumerate}
        \item A \emph{$\mathbf{T}$-algebra} is a BAO $(B,\Diamond)$ such that $a \le \Diamond a$ for each $a \in B$. 
        \item A \emph{$\mathbf{K4}$-algebra} is a BAO $(B,\Diamond)$ such that $\Diamond\Diamond a \le \Diamond a$ for each $a \in B$.
        \item A \emph{$\mathbf{B}$-algebra} is a BAO $(B,\Diamond)$ such that $\Diamond\neg\Diamond a \le\neg a$ for each $a \in B$.
        \item A \emph{$\mathbf{TB}$-algebra} is a BAO which is both a $\mathbf{T}$-algebra and a $\mathbf{B}$-algebra.
        \item An \emph{$\mathbf{S5}$-algebra} is a BAO which is both a closure algebra and a $\mathbf{B}$-algebra.
    \end{enumerate}
\end{definition}
As the name suggests, $\mathbf{T}$-algebras serve as algebraic models of the modal system $\mathbf{T}$, and the same applies to $\mathbf{K4}$-algebras, $\mathbf{B}$-algebras, $\mathbf{TB}$-algebras, and $\mathbf{S5}$-algebras (the latter were introduced by Halmos \cite{Hal56} under the name of monadic algebras).

\begin{definition} \label{def: categories of BAOs} \
    \begin{enumerate}
        \item Let $\bf BAO_{st}$ denote the category of BAOs and stable morphisms between them.
        \item Let $\CA$ denote the full subcategory of $\BAO$ consisting of closure algebras.
        \item Let $\T$, $\Kfour$, $\B$, $\TB$, and $\MA$ denote the full subcategories of $\BAO$ consisting of the algebras introduced in Definition~\ref{def: algebras}.
    \end{enumerate}
\end{definition}

We have the following inclusion relationships between these categories:

\[\begin{tikzcd}[sep=small, column sep=tiny]
	&& \BAO \\[6ex]
	\\
	\Kfour && \T && \B \\
	\\
	& \CA && \TB \\
	\\
	&& \MA
	\arrow[hook, from=3-1, to=1-3]
	\arrow[hook, from=3-3, to=1-3]
	\arrow[hook', from=3-5, to=1-3]
	\arrow[hook', from=5-2, to=3-1]
	\arrow[hook, from=5-2, to=3-3]
	\arrow[hook', from=5-4, to=3-3]
	\arrow[hook, from=5-4, to=3-5]
	\arrow[hook', from=7-3, to=5-2]
	\arrow[hook, from=7-3, to=5-4]
\end{tikzcd}\]

We let $\CBAO$ denote the category of complete BAOs with complete stable morphisms between them, and define $\CT$, $\CKfour$, $\CB$, $\CTB$, and $\CMA$ similarly (observe that ${\bf CCA_{st}} = \MT$). We have the following analog of Theorem~\ref{thm: MT lack copowers}.

\begin{theorem} \label{thm: lack of copowers}
    Each of the categories $\CBAO$, $\CT$, $\CKfour$, $\CB$, $\CTB$, and $\CMA$ lacks some countable copowers. Consequently, none is equivalent to a prevariety.
\end{theorem}

\begin{proof}
    Let $\C$ be one of the categories in the statement. Observe that the four-element boolean algebra equipped with $\Diamond_s$ belongs to $\C$. Thus, in view of Theorem~\ref{thm: MT not cocomplete}, it is sufficient to show that the forgetful functor $\mathcal{U} \colon \C \to \CBA$ is both a left and right adjoint. The same proof as in Theorem~\ref{t:adjoints} gives a functor $\mathcal{L} \colon \CBA \to \C$ that is left adjoint to $\mathcal{U}$. To define $\mathcal{R} \colon \CBA \to \C$, we consider cases. 
    
    First suppose that $\C$ is $\CT$, $\CTB$, or $\CMA$. Then the same proof as in Theorem~\ref{t:adjoints} gives a functor $\mathcal{R} \colon \CBA \to \C$ that is right adjoint to $\mathcal{U}$ (since $(A,\Diamond_i) \in \C$ for any $A \in \CBA$). Next suppose that $\C$ is $\CBAO$, $\CB$, or $\CKfour$. Then define $\mathcal{R} \colon \CBA \to \C$ by $\mathcal{R}(A)=(A,\Diamond_0)$, where $\Diamond_0 \colon A \to A$ is given by $\Diamond_0 a = 0$ for each $a\in A$, and on morphisms $\mathcal{R}$ is the identity.
    Since $\mathcal R(A)$ is both a $\mathbf{K4}$-algebra and a $\mathbf{B}$-algebra, $\mathcal R$ is well defined. Moreover, for $A\in\CBA$ and $B\in\C$, we have
    \[
    \hom_{\C}(B, \mathcal{R}(A)) = \hom_{\CBA}(\mathcal{U}(B),A).
    \]
    Indeed, the inclusion $\subseteq$ is clear since each $\C$-morphism is a complete boolean morphism.
    For the reverse inclusion, suppose $g \colon \mathcal{U}(B) \to A$ is a $\CBA$-morphism.
    Then it is a stable morphism from $B$ to $\mathcal{R}(A)$ because, for every $a \in B$,
    \[
    \Diamond_0 g(a) = 0 \leq g(\Diamond a),
    \]
    yielding that $g$ is a $\C$-morphism. Thus, $\mathcal{R} \colon \CBA \to \C$ is right adjoint to $\mathcal{U}$.
\end{proof}

\section{MT-algebras with stable morphisms}\label{sec: cBA}

In this section we show that the situation remains the same if instead of complete stable morphisms we consider all stable morphisms.
We do this by showing that binary powers don't exist on the dual side.
We start with the following general result.

\begin{theorem} \label{thm: general}
    Let $\C$ be a complete category.
    Suppose $\pazocal{H}$ is a class of objects of $\C$ such that every morphism $f \colon X \to Y$ in $\C$ is an isomorphism if, for each $Z \in \pazocal{H}$, $f \circ - \colon \hom_\C(Z, X) \to \hom_\C(Z, Y)$ is a bijection.
    Then the inclusion $\mathbf{E} \hookrightarrow \C$ of any full subcategory $\mathbf{E}$ containing all objects of $\pazocal{H}$ preserves limits.
\end{theorem}

\begin{proof}
    Let $D \colon \mathbf{I} \to \mathbf{E}$ be a diagram and $(m_i \colon M \to D(i))_{i \in \mathbf{I}}$ a limit cone in $\mathbf{E}$.
    Suppose $(n_i \colon N \to D(i))_{i \in \mathbf{I}}$ is the limit in $\mathbf{C}$, which exists because $\C$ is complete.
    Then there is a unique morphism $f \colon M \to N$ such that for every $i \in \mathbf{I}$ we have $m_i = n_i \circ f$.
    Since $(m_i \colon M \to D(i))_{i \in \mathbf{I}}$, considered as a cone in $\C$, is obtained by composing the limit cone $(n_i \colon N \to D(i))_{i \in \mathbf{I}}$ with $f$ and composing a limit cone with an isomorphism gives a limit cone, to prove that $(m_i \colon M \to D(i))_{i \in \mathbf{I}}$ is a limit cone in $\C$, it is enough to prove that $f$ is an isomorphism.
    
    Fix $Z \in \pazocal{H}$.
    The functors $\hom_\E(Z, -) \colon \E \to \Set$ and $\hom_\C(Z, -) \colon \C \to \Set$ preserve limits because every representable functor does \cite[Cor.~13.9]{AHS06}.\footnote{We recall that a representable functor is a functor that up to natural isomorphism is of the form $\hom(A,-)$ for some object $A$.}
    Therefore,
    \[
    (m_i \circ - \colon \hom_\E(Z,M) \to \hom_\E(Z,D(i)))_{i \in \mathbf{I}}
    \]
    is a limit cone in $\Set$ 
    over the diagram $\hom_\E(Z,-) \circ D \colon \mathbf{I} \to \Set$
    and 
    \[
    (n_i \circ - \colon \hom_\C(Z,N) \to \hom_\C(Z,D(i)))_{i \in \mathbf{I}}
    \]
    is a limit cone in $\Set$ over the diagram $\hom_\C(Z,-) \circ D  \colon \mathbf{I} \to \Set$.
    Since $\E$ is a full subcategory of $\C$, for every $Z  \in \pazocal{H}$ we have that $\hom_\E(Z,M) = \hom_\C(Z,M)$ and $\hom_\E(Z,D(i)) = \hom_\C(Z,D(i))$ for each $i \in \mathbf{I}$. 
    Moreover, we have a function $f \circ -\colon \hom_\C(Z,M) \to \hom_\C(Z, N)$ that maps ${h \colon Z \to M}$ to $f \circ h \colon Z \to N$. We claim that $f \circ -$ is a bijection.
    Since $m_i = n_i \circ f$ for every $i$, the function 
    \[
    m_i \circ - \colon \hom_\C(Z,M) \to \hom_\C(Z,D(i))
    \]
    is the composite 
    \[
        \begin{tikzcd}[column sep=10ex]
            \hom_\C(Z,M) \arrow[r, "f \circ -"] & \hom_\C(Z, N) \arrow[r, "n_i \circ -"] & \hom_\C(Z,D(i)).
        \end{tikzcd}
    \]
    Therefore, $f \circ -$ is the unique factorization of the limit cone
    \[
        (m_i \circ - \colon \hom_\C(Z,M) \to \hom_\C(Z,D(i)))_{i \in \mathbf{I}}
    \]
    through the limit cone
    \[
        (n_i \circ - \colon \hom_\C(Z,N) \to \hom_\C(Z,D(i)))_{i \in \mathbf{I}}.
    \]
    Because both cones are limit cones over the same diagram, $f \circ -$ must be an isomorphism in $\Set$, hence a bijection.
    Thus, it is enough to use the assumption to conclude that $f$ is an isomorphism in $\C$.
\end{proof}

We recall that a \emph{conservative} functor is a functor $\mathcal F \colon \mathbf{A} \to \mathbf{B}$ that reflects isomorphisms, i.e.~such that, for every morphism $f$ in $\mathbf{A}$, if $\mathcal F(f)$ is an isomorphism in $\mathbf{B}$ then $f$ is an isomorphism in $\mathbf{A}$ (see, e.g., \cite[Sec.\ 5.6]{Riehl2016}).

\begin{corollary} \label{cor: general-conservative}
    Let $\C$ be a category with a conservative functor to $\Set$ that is represented by an object $X \in \C$.
    The inclusion $\mathbf{E} \hookrightarrow \C$ of any full subcategory $\mathbf{E}$ containing $X$ preserves limits.
\end{corollary}

\begin{proof}
    This is the specialization of Theorem~\ref{thm: general} to the case of $\pazocal{H}=\{X\}$.
\end{proof}

Let $\KHaus$ denote the category of compact Hausdorff spaces and continuous maps. 

\begin{theorem} \label{t:limits-as-in-KH}
    Let $\mathbf{E}$ be a full subcategory of $\KHaus$ containing a singleton. The inclusion $\mathbf{E} \hookrightarrow \KHaus$ preserves limits.
\end{theorem}

\begin{proof}
    We show that $\mathbf{E}$ and $\KHaus$ satisfy the condition of Corollary~\ref{cor: general-conservative} by taking $\mathbf{C} = \KHaus$. Since a continuous map between compact Hausdorff spaces is a homeomorphism if and only if it is a bijection, the forgetful functor $\mathcal U \colon\KHaus \to \Set$ is conservative. Let $\{x\} \in \mathbf{E}$ be a singleton space. There is a natural isomorphism between $\hom_\KHaus(\{x\},-)$ and $\mathcal U$ whose component relative to $Y \in \KHaus$ sends $f \colon \{x\} \to Y$ to $f(x)$. Thus, $\mathcal U \colon \KHaus \to \Set$ is a conservative functor represented by $\{x\}$. Since $\mathbf{E}$ contains $\{x\}$, the result follows from Corollary~\ref{cor: general-conservative}.
\end{proof}

We recall that a {\em Stone space} is a zero-dimensional compact Hausdorff space. Let $\Stone$ denote the full subcategory of $\KHaus$ consisting of Stone spaces. By the celebrated Stone duality, $\BA$ is dually equivalent to $\Stone$.
We call an extremally disconnected compact Hausdorff space an \emph{ED-space}. It is straightforward to see that each ED-space is a Stone space. Let $\ED$ denote the full subcategory of $\Stone$ consisting of ED-spaces.

\begin{theorem}\label{thm: ED lacks products}
    The category $\ED$ lacks some binary powers.
\end{theorem}

\begin{proof}
    This is an immediate consequence of Theorem~\ref{t:limits-as-in-KH} and the well-known fact in topology that there are ED-spaces whose product is not extremally disconnected (for example, $\beta\N$ is an ED-space but $\beta\N\times\beta\N$ is not; see \cite[p.~97]{GJ60}). 
\end{proof}

Let $\cBA$ denote the category of complete boolean algebras with boolean morphisms. Dualizing Theorem~\ref{thm: ED lacks products}, we get: 

\begin{corollary} \label{c:lacking-binary-coproducts}
    $\cBA$ lacks some binary copowers, and hence is not equivalent to a prevariety.
\end{corollary}

\begin{proof}
    Since a boolean algebra is complete if and only if its Stone space is extremally disconnected (see, e.g., \cite[Thm.~39]{GH09}), restricting Stone duality to complete boolean algebras yields that $\cBA$ is dually equivalent to $\ED$. 
    Therefore, since $\ED$ lacks some binary powers, $\cBA$ lacks some binary copowers. Thus, $\cBA$ is not cocomplete, and hence is not equivalent to a prevariety. 
\end{proof} 

As an immediate consequence of Corollary~\ref{c:lacking-binary-coproducts}, we obtain:

\begin{corollary} \label{cor:no-bin-prod-if-surjects}
    A category $\C$ lacks some binary copowers provided there is a colimit-preserving essentially surjective functor $\mathcal U \colon \C \to \cBA$.
\end{corollary}

Let $\MTst$ be the category of MT-algebras and stable morphisms. We have the following version of Theorem~\ref{thm: MT lack copowers} for $\MTst$: 

\begin{corollary}\label{Cor: MTst not cocomplete}
    $\MTst$ lacks some binary copowers, and hence is not equivalent to a prevariety.
\end{corollary}

\begin{proof} 
    The same argument as in the proof of Theorem~\ref{t:adjoints} shows that the forgetful functor $\mathcal U\colon\MTst\to\cBA$ is both a left and right adjoint. The former gives that $\mathcal U$ is colimit-preserving and it is clearly essentially surjective. Thus, $\MTst$ lacks some binary copowers by Corollary~\ref{cor:no-bin-prod-if-surjects}.
\end{proof}

We finish this section by considering the following analogs of the categories in Definition \ref{def: categories of BAOs}. 
Let $\cBAO$ denote the category whose objects are complete BAOs and whose morphisms are stable morphisms, and define $\cT$, $\cKfour$, $\cB$, $\cTB$, and $\cMA$ similarly (observe that ${\bf cCA_{st}} = \MTst$). The same proof as above yields the following version of Theorem~\ref{thm: lack of copowers}:

\begin{corollary}
    Each of the categories $\cBAO$, $\cT$, $\cKfour$, $\cB$, $\cTB$, and $\cMA$ lacks some binary copowers. Consequently, none of them is equivalent to a prevariety.
\end{corollary}

\section{Closure algebras with stable morphisms}\label{sec: 3.2}

In Section~\ref{sec: CBAOs} we showed that the category $\MT$ of MT-algebras and MT-morphisms lacks some countable copowers. In Section~\ref{sec: cBA} we relaxed the notion of an MT-morphism by dropping the completeness assumption and showed that the category $\MTst$ of MT-algebras and stable morphisms lacks some binary copowers. In this section, we further drop the requirement of completeness on the object side and show that the category $\CA$ of closure algebras and stable morphisms lacks some coequalizers. This we do by showing that the corresponding category of Stone spaces equipped with continuous quasi-orders lacks some equalizers and then employing J\'onsson-Tarski duality. The same results are proved for several other categories of BAOs with stable morphisms. Consequently, none of these is equivalent to a prevariety.

Let $X$ and $X'$ be compact Hausdorff spaces, $R$ a closed relation on $X$ (i.e., a closed subset of $X \times X$), and $R'$ a closed relation on $X'$. Following \cite[Def.~3.2]{BBI16}, we call a continuous map $f\colon X\to X'$ \emph{stable} provided 
\[
    x \mathrel{R} y \Longrightarrow f(x) \mathrel{R'} f(y) \quad \forall x,y\in X.
\]

\begin{definition}
    Let $\KHausR$ denote the category whose objects are pairs $(X,R)$ where $X$ is a compact Hausdorff space and $R$ is a closed relation on $X$, and whose morphisms are stable maps.
\end{definition}

It is clear that isomorphisms in $\KHausR$ are continuous bijections that preserve and reflect the binary relations.

\begin{proposition}\label{Prop: KHausR complete}
    The category $\KHausR$ is complete.
\end{proposition}
    
\begin{proof}
    The product of $\{(X_i,R_i) : i \in I\}$ in $\KHausR$ is $(\prod X_i, \prod R_i)$, where $\prod X_i$ is the product of $\{X_i : i \in I\}$ in $\KHaus$ and $\prod R_i$ is the product relation given by 
    \[
    f \left(\mathrel{\prod R_i}\right) g \iff f(i) \mathrel{R_i} g(i) \ \ \forall i \in I.
    \]
    The equalizer of $f, g \colon (X,R) \to (X',R')$ in $\KHausR$ is their equalizer 
    \[
    \{ x \in X : f(x) = g(x)\}
    \]
    in $\KHaus$ equipped with the restriction of $R$. Consequently, $\KHausR$ is complete.
\end{proof}

Let $\tworef = \{ 0,1 \}$ denote the two-element chain with its usual (reflexive) order, which we view as an object of $\KHausR$ by equipping it with the discrete topology. 

\begin{theorem} 
    \label{thm: preservation limits: reflexive}
        If $\E$ is a full subcategory of $\KHausR$ containing $\tworef$ and such that for all $(X, R) \in \E$ the relation $R$ is reflexive, then the inclusion $\E \hookrightarrow \KHausR$ preserves limits.
\end{theorem}

\begin{proof}
    Let $\KHausT$ denote the full subcategory of $\KHausR$ consisting of those $(X,R)$ where $R$ is reflexive.\footnote{The subscript $\mathbf{T}$ is motivated by Theorem~\ref{Thm: JT correspondence}.}
    Since the inclusion $\KHausT \hookrightarrow \KHausR$ preserves limits (by the proof of Proposition~\ref{Prop: KHausR complete}), it is enough to prove that the inclusion $\E \hookrightarrow \KHausT$ preserves limits.
    This can be done by verifying the condition of Corollary~\ref{cor: general-conservative} by showing that
    \[
    \hom_{\KHausT}(\tworef,-) \colon \KHausT \to \Set
    \]
    is conservative. Let $f \colon (X,R) \to (X',R')$ be a morphism in $\KHausT$ whose image
    \[
    f \circ - \colon \hom_{\KHausT}(\tworef,X) \to \hom_{\KHausT}(\tworef, X')
    \]
    is a bijection. 
    To prove that $f$ is an isomorphism in $\KHausT$ it is sufficient to show that it is a bijection that reflects the relation. We first show that $f$ is a bijection. Let $y \in X'$, and consider the map $c_y \colon  \tworef \to X'$ with constant value $y$, which is a morphism in $\KHausT$. Since $f \circ -$ is onto, there is a morphism $g \colon \tworef \to X$ in $\KHausT$ such that $f \circ g = c_y$. Therefore, $y$ is in the image of $f$, and so $f$ is onto. Suppose that $x_1,x_2 \in X$ are such that $f(x_1)=f(x_2)$ and consider the maps $c_{x_1}, c_{x_2} \colon  \tworef \to X$ with constant values $x_1,x_2$. Then $f \circ c_{x_1} = f \circ c_{x_2}$ because $f(x_1)=f(x_2)$. Since $f \circ -$ is one-to-one, $c_{x_1}=c_{x_2}$, and so $x_1=x_2$. Thus, $f$ is one-to-one, and hence $f$ is a bijection.

    We now show that $f$ reflects the relation. Let $x_1,x_2 \in X$ be such that $f(x_1) \mathrel{R'} f(x_2)$. Define $h \colon \tworef \to X'$ by $h(0) = f(x_1)$ and $h(1) = f(x_2)$. Because $\tworef$ is discrete, $h$ is continuous. Since ${f(x_1) \mathrel{R'} f(x_2)}$, we have that $h$ is a morphism in $\KHausT$. Because $f \circ -$ is onto, there is a morphism $g \colon \tworef \to X$ such that $f \circ g = h$. So, $fg(0)=h(0)=f(x_1)$ and $fg(1)=h(1)=f(x_2)$. Since $f$ is one-to-one, we conclude that $g(0)=x_1$ and $g(1)=x_2$. Thus, $x_1 \mathrel{R} x_2$ because $0 < 1$ and $g$ is a morphism in $\KHausT$.
\end{proof}

We recall that a closed relation $R$ on a compact Hausdorff space $X$ is \emph{continuous} provided 
\[
R^{-1}[U] \coloneqq \{ x \in X : x \mathrel{R} u \mbox{ for some } u \in U\}
\]
is open for each open $U \subseteq X$. 

\begin{definition}\label{def: Cat with cont R}
    Let $\StoneCSfour$ denote the full subcategory of $\KHausR$ whose objects are those $(X,R)$ where $X$ is a Stone space and $R$ is a continuous quasi-order \textup{(}reflexive and transitive\textup{)}.
\end{definition}

Applying Theorem~\ref{thm: preservation limits: reflexive} yields:

\begin{corollary}\label{cor: inclusion StoneCfour preserves limits}
    The inclusion $\StoneCSfour \hookrightarrow \KHausR$ preserves limits.
\end{corollary}

We prove that $\StoneCSfour$ is not complete.

\begin{theorem} \label{thm: continuous StoneC}
    $\StoneCSfour$ lacks some equalizers.
\end{theorem}

\begin{proof}
    Consider the space $X$ shown in Figure~\ref{fig1}, where $x_\infty$ is the limit of the sequence $(x_n)$ of isolated points and $y_\infty$ is the limit of the sequence $(y_n)$ of isolated points. Therefore, $X$ is a Stone space.
    The relation $R$ is defined on $X$ as shown in Figure~\ref{fig1}, where circles indicate that $R$ is reflexive. It is not difficult to see that $R$ is a continuous relation.
    Thus, $(X,R)$ is an object of $\StoneCSfour$.
    \begin{figure}[h]
        \centering
       \begin{tikzpicture}
        	\begin{pgfonlayer}{nodelayer}
        		\node [style=white dot] (0) at (-3, 0.5) {};
        		\node [style=white dot] (1) at (-3, 2) {};
        		\node [style=white dot] (2) at (-2, 2) {};
        		\node [style=white dot] (3) at (-2, 0.5) {};
        		\node [style=white dot] (4) at (-1, 0.5) {};
        		\node [style=white dot] (5) at (-1, 2) {};
        		\node [style=white dot] (6) at (2, 2) {};
        		\node [style=white dot] (7) at (2, 0.5) {};
        		\node [style=white dot] (8) at (-3, 0.5) {};
        		\node [style=small black dot] (9) at (0, 2) {};
        		\node [style=small black dot] (10) at (0.5, 2) {};
        		\node [style=small black dot] (11) at (1, 2) {};
        		\node [style=small black dot] (12) at (0, 0.5) {};
        		\node [style=small black dot] (13) at (0.5, 0.5) {};
        		\node [style=small black dot] (14) at (1, 0.5) {};
        		\node [style=none] (19) at (-3, -0.25) {$x_0$};
        		\node [style=none] (20) at (-2, -0.25) {$x_1$};
        		\node [style=none] (21) at (-1, -0.25) {$x_2$};
        		\node [style=none] (22) at (2, -0.25) {$x_\infty$};
        		\node [style=none] (23) at (-3, 2.75) {$y_0$};
        		\node [style=none] (24) at (-2, 2.75) {$y_1$};
        		\node [style=none] (25) at (-1, 2.75) {$y_2$};
        		\node [style=none] (26) at (2, 2.75) {$y_\infty$};
        	\end{pgfonlayer}
        	\begin{pgfonlayer}{edgelayer}
        		\draw[style=to] (8) to (1);
        		\draw[style=to] (3) to (2);
        		\draw[style=to] (4) to (5);
        		\draw[style=to] (7) to (6);
        	\end{pgfonlayer}
        \end{tikzpicture}
        \caption{The space $(X,R)$.}
        \label{fig1}
    \end{figure}
        
    Let $X'$ be the space shown in Figure~\ref{fig2}, where $u_\infty$ is the limit of the sequence $(u_n)$ of isolated points and $v_\infty$ is the limit of each of the sequences $(v_n)$ and $(w_n)$ of isolated points. Therefore, $X'$ is a Stone space. The relation $R'$ is defined on $X'$ as shown in Figure~\ref{fig2}, where circles indicate that $R'$ is reflexive. It is not difficult to see that $R'$ is a continuous relation.
    Thus, $(X',R')$ is an object of $\StoneCSfour$.
    \begin{figure}[!h]
        \centering
        \begin{tikzpicture}
        	\begin{pgfonlayer}{nodelayer}
        		\node [style=white dot] (1) at (-4.5, 2) {};
        		\node [style=white dot] (2) at (-2.75, 2) {};
        		\node [style=white dot] (3) at (-3.25, 0.5) {};
        		\node [style=white dot] (4) at (-1.5, 0.5) {};
        		\node [style=white dot] (5) at (-1, 2) {};
        		\node [style=white dot] (6) at (2, 2) {};
        		\node [style=white dot] (7) at (2, 0.5) {};
        		\node [style=white dot] (8) at (-5, 0.5) {};
        		\node [style=small black dot] (9) at (0, 2) {};
        		\node [style=small black dot] (10) at (0.5, 2) {};
        		\node [style=small black dot] (11) at (1, 2) {};
        		\node [style=small black dot] (12) at (0, 0.5) {};
        		\node [style=small black dot] (13) at (0.5, 0.5) {};
        		\node [style=small black dot] (14) at (1, 0.5) {};
        		\node [style=none] (19) at (-5, -0.25) {$u_0$};
        		\node [style=none] (20) at (-3.25, -0.25) {$u_1$};
        		\node [style=none] (21) at (-1.5, -0.25) {$u_2$};
        		\node [style=none] (22) at (2, -0.25) {$u_\infty$};
        		\node [style=none] (23) at (-5.75, 2.5) {$v_0$};
        		\node [style=none] (24) at (-3.75, 2.5) {$v_1$};
        		\node [style=none] (25) at (-2, 2.5) {$v_2$};
        		\node [style=none] (26) at (2, 2.5) {$v_\infty$};
        		\node [style=white dot] (27) at (-5.5, 2) {};
        		\node [style=white dot] (28) at (-3.75, 2) {};
        		\node [style=white dot] (30) at (-2, 2) {};
        		\node [style=none] (31) at (-4.5, 2.5) {$w_0$};
        		\node [style=none] (32) at (-2.75, 2.5) {$w_1$};
        		\node [style=none] (33) at (-1, 2.5) {$w_2$};
        	\end{pgfonlayer}
        	\begin{pgfonlayer}{edgelayer}
        		\draw[style=to] (8) to (1);
        		\draw[style=to] (3) to (2);
        		\draw[style=to] (4) to (5);
        		\draw[style=to] (7) to (6);
        		\draw[style=to] (8) to (27);
        		\draw[style=to] (3) to (28);
        		\draw[style=to] (4) to (30);
        	\end{pgfonlayer}
        \end{tikzpicture}
        \caption{The space $(X',R')$.}
        \label{fig2}
    \end{figure}

    Define $f,g\colon X \to X'$ as follows:
    \begin{align*}
        f(x_\infty) ={} & u_\infty = g(x_\infty),\\
        f(x_n) ={} & u_n = g(x_n),\\
        f(y_\infty) ={} & v_\infty = g(y_\infty),\\
        f(y_n) ={} & v_n,\\
        & w_n = g(y_n). 
    \end{align*}
    In other words, $f$ and $g$ only differ in that $f(y_n)=v_n$ while $g(y_n) = w_n$.
    It is straightforward to see that $f$ and $g$ are $\KHausR$-morphisms, and that the equalizer of $f$ and $g$ in $\KHausR$ is $Y = \{ x_n : n \in \mathbb N\} \cup \{ x_\infty,y_\infty \}$ (see the proof of Proposition~\ref{Prop: KHausR complete}).
    But the restriction of $R$ to $Y$ is no longer continuous because $\{ y_\infty \}$ is open in $Y$, but $R^{-1}[\{y_\infty\}] = \{ y_\infty,x_\infty \}$ is not.
    Thus, by Corollary~\ref{cor: inclusion StoneCfour preserves limits},
    the equalizer of $f$ and $g$ does not exist in $\StoneCSfour$.
\end{proof}

We now apply J\'onsson-Tarski duality for BAOs \cite{JT51} (see also \cite{CZ97} or \cite{BRV01}) to closure algebras, by which a BAO $(B,\Diamond)$ is a closure algebra if and only if in its J\'onsson-Tarski dual $(X,R)$ the relation $R$ is reflexive and transitive. This together with the fact that stable morphisms between BAOs correspond to stable maps between their dual spaces \cite[Lem.~3.3]{BBI16} yields that $\CA$ is dually equivalent to $\StoneCSfour$.
Putting this together with Theorem~\ref{thm: continuous StoneC} immediately gives:

\begin{corollary}\label{Cor: CA not cocomplete}
    $\CA$ lacks some coequalizers, and hence it is not equivalent to a prevariety.
\end{corollary}

The table below summarizes the main results of Sections~\ref{sec: CBAOs} to \ref{sec: 3.2}.

\begin{table}[h]
    \centering
    \begin{tabular}{|l|l|l|l|}
        \hline
        \textbf{Category} & \textbf{Objects} & \textbf{Morphisms} & \textbf{Location} \\
        \hline
        $\MT$ & MT-algebras & MT-morphisms & Thm.~\ref{thm: MT lack copowers}\\
        $\MTst$ & MT-algebras & Stable morphisms & Cor.~\ref{Cor: MTst not cocomplete} \\
        $\CA$ & Closure algebras & Stable morphisms & Cor.~\ref{Cor: CA not cocomplete} \\
        \hline
    \end{tabular}
    \vspace{3mm}
    \caption{Categories of closure algebras that are not cocomplete.}\label{table new1}
\end{table}

As in the previous two sections, Corollary~\ref{Cor: CA not cocomplete} extends to various categories of BAOs with stable morphisms. 
To do so, as we did for $\CA$, we use duality.
For this, we recall the following well-known result (see \cite{JT51}, \cite{CZ97}, or \cite{BRV01}).

\begin{theorem}\label{Thm: JT correspondence}
    Let $(B,\Diamond)$ be a BAO and $(X,R)$ its J\'onsson-Tarski dual. 
    \begin{enumerate}
        \item $(B,\Diamond)$ is a $\mathbf{T}$-algebra if and only if $R$ is reflexive;
        \item $(B,\Diamond)$ is a $\mathbf{K4}$-algebra if and only if $R$ is transitive;
        \item $(B,\Diamond)$ is a $\mathbf{B}$-algebra if and only if $R$ is symmetric;
        \item $(B,\Diamond)$ is a $\mathbf{TB}$-algebra if and only if $R$ is reflexive and symmetric;
        \item $(B,\Diamond)$ is an $\mathbf{S5}$-algebra if and only if $R$ is an equivalence relation.
    \end{enumerate}
\end{theorem}

We thus have the following counterpart of Definition~\ref{def: categories of BAOs}.  

\begin{definition}
    Let $\StoneC$ denote the full subcategory of $\KHausR$ whose objects are those $(X,R)$ where $X$ is a Stone space and $R$ is a continuous relation, and define 
    $\StoneCT$, $\StoneCKfour$, $\StoneCB$, $\StoneCTB$, and $\StoneCSfive$ as the full subcategories of $\StoneC$ with the additional condition on $R$ described in Theorem~\ref{Thm: JT correspondence}.
\end{definition}

We have the following inclusion relationship between the categories defined above and $\StoneCSfour$: 

\[\begin{tikzcd}[sep=small, column sep=-1.3em]
	&& \StoneC \\[6ex]
	\\
	\StoneCKfour && \StoneCT && \StoneCB \\
	\\
	& \StoneCSfour && \StoneCTB \\
	\\
	&& \StoneCSfive
	\arrow[hook, from=3-1, to=1-3]
	\arrow[hook, from=3-3, to=1-3]
	\arrow[hook', from=3-5, to=1-3]
	\arrow[hook', from=5-2, to=3-1]
	\arrow[hook, from=5-2, to=3-3]
	\arrow[hook', from=5-4, to=3-3]
	\arrow[hook, from=5-4, to=3-5]
	\arrow[hook', from=7-3, to=5-2]
	\arrow[hook, from=7-3, to=5-4]
\end{tikzcd}\]

Theorem~\ref{Thm: JT correspondence} together with the fact that stable morphisms between BAOs correspond to stable maps between their dual spaces yields the following version of J\'onsson-Tarski duality and its restrictions:

\begin{theorem}\ \label{thm: JT for BAOs}
\hfill    
\begin{enumerate}
        \item $\BAO$ is dually equivalent to $\StoneC$.
        \item $\T$ is dually equivalent to $\StoneCT$.
        \item $\Kfour$ is dually equivalent to $\StoneCKfour$.
        \item $\B$ is dually equivalent to $\StoneCB$.
        \item $\TB$ is dually equivalent to $\StoneCTB$.
        \item $\MA$ is dually equivalent to $\StoneCSfive$.
    \end{enumerate}
\end{theorem}

We now show that, like $\StoneCSfour$, all the categories in the diagram above also lack some equalizers. We follow the convention in modal logic (see \cite[p.~66]{CZ97}) and denote reflexive points of a binary relation by 
$\tikz[baseline=-0.5ex]{
\node[circle,draw,inner sep=2pt,fill=white] (a) at (0,0) {};
}$
and irreflexive points by
$\tikz[baseline=-0.5ex]{
\node[circle,draw,inner sep=2pt,fill=black] (a) at (0,0) {};
}$. Let $\oneirr$ denote the irreflexive singleton and $\twoirr$ the irreflexive 2-element chain
$\tikz[baseline=-0.5ex]{
  \node[circle,draw,inner sep=2pt,fill=black] (a) at (0,0) {};
  \node[circle,draw,inner sep=2pt,fill=black] (b) at (0.7,0) {};
  \draw[->, shorten <=0.5pt, shorten >=0.5pt] (a) -- (b);
}
$. 
We denote the bottom of this chain by $0$ and the top by $1$.
Finally, let 
$\twoirrsym$ denote the irreflexive doubleton
$\tikz[baseline=-0.5ex]{
  \node[circle,draw,inner sep=2pt,fill=black] (a) at (0,0) {};
  \node[circle,draw,inner sep=2pt,fill=black] (b) at (0.7,0) {};
  \draw[<->, shorten <=0.5pt, shorten >=0.5pt] (a) -- (b);
}$ and 
$\tworefsym$ the reflexive doubleton
$\tikz[baseline=-0.5ex]{
  \node[circle,draw,inner sep=2pt,fill=white] (a) at (0,0) {};
  \node[circle,draw,inner sep=2pt,fill=white] (b) at (0.7,0) {};
  \draw[<->, shorten <=0.5pt, shorten >=0.5pt] (a) -- (b);
}$. In each of these cases, the two points (which we also denote by $0$ and $1$) are both related to each other.
We view $\oneirr$, $\twoirr$, $\twoirrsym$, and $\tworefsym$ as objects of $\KHausR$ by equipping them with the discrete topology.
The next theorem is an analog of Theorem~\ref{thm: preservation limits: reflexive}.

\begin{theorem}\label{thm: preservation limits}
    Let $\E$ be a full subcategory of $\KHausR$ satisfying one of the following conditions.
    \begin{enumerate}
    
        \item\label{thm: preservation limits: irreflexive}
        $\E$ contains $\oneirr$ and $\twoirr$.
        
        \item\label{thm: preservation limits: irreflexive symmetric}
        For all $(X,R) \in \E$ the relation $R$ is symmetric, and $\E$ contains $\oneirr$ and $\twoirrsym$.
        
        \item\label{thm: preservation limits: reflexive symmetric} 
        For all $(X,R) \in \E$ the relation $R$ is reflexive and symmetric, and $\E$ contains $\tworefsym$.
        
    \end{enumerate}
    Then the inclusion functor $\E \hookrightarrow \KHausR$ preserves limits.
\end{theorem}

\begin{proof}
    \eqref{thm: preservation limits: irreflexive}.
    Let $\C = \KHausR$ and $\pazocal{H} = \{ \oneirr,\twoirr \}$. 
    We prove that the condition of Theorem~\ref{thm: general} is satisfied by showing that, for any morphism $f \colon (X,R) \to (X',R')$ in $\KHausR$, if the functions
    \[
    f \circ - \colon \hom_{\KHausR}(\oneirr,X)  \to \hom_{\KHausR}(\oneirr,X')
    \]
    and
    \[
    f \circ - \colon \hom_{\KHausR}(\twoirr,X)  \to \hom_{\KHausR}(\twoirr,X')
    \]
    are bijections, then $f$ is an isomorphism.
    To prove that $f$ is an isomorphism in $\KHausR$ it is sufficient to show that it is a bijection that reflects the relation, which we do by adjusting the proof of Theorem~\ref{thm: preservation limits: reflexive} accordingly. 
    
    We first show that $f$ is a bijection. Let $y \in X'$, and consider the map $c_y \colon  \oneirr \to X'$ with constant value $y$, which is a morphism in $\KHausR$. Since $f \circ -$ is onto, there is a morphism $g \colon \oneirr \to X$ in $\KHausR$ such that $f \circ g = c_y$. Therefore, $y$ is in the image of $f$, and so $f$ is onto. Suppose that $x_1,x_2 \in X$ are such that $f(x_1)=f(x_2)$ and consider the maps $c_{x_1}, c_{x_2} \colon  \oneirr \to X$ with constant values $x_1,x_2$. Then $f \circ c_{x_1} = f \circ c_{x_2}$ because $f(x_1)=f(x_2)$. Since $f \circ -$ is one-to-one, $c_{x_1}=c_{x_2}$, and so $x_1=x_2$. Thus, $f$ is one-to-one, and hence $f$ is a bijection.

    We now show that $f$ reflects the relation. Let $x_1,x_2 \in X$ be such that $f(x_1) \mathrel{R'} f(x_2)$. Define $h \colon \twoirr \to X'$ by $h(0) = f(x_1)$ and $h(1) = f(x_2)$. Since $\twoirr$ is discrete, $h$ is continuous. Because ${f(x_1) \mathrel{R'} f(x_2)}$, we have that $h$ is a morphism in $\KHausR$. Since $f \circ -$ is onto, there is a morphism $g \colon \twoirr \to X$ such that $f \circ g = h$. Therefore,  $fg(0)=h(0)=f(x_1)$ and $fg(1)=h(1)=f(x_2)$. Since $f$ is one-to-one, we conclude that $g(0)=x_1$ and $g(1)=x_2$. Thus, $x_1 \mathrel{R} x_2$ because $0 < 1$ and $g$ is a morphism in $\KHausR$.
    
    \eqref{thm: preservation limits: irreflexive symmetric}.
    Let $\KHausB$ denote the full subcategory of $\KHausR$ consisting of those $(X,R)$ where $R$ is symmetric.
    It follows from the proof of Proposition~\ref{Prop: KHausR complete} that the inclusion $\KHausB \hookrightarrow \KHausR$ preserves limits. Therefore, it is enough to prove that the inclusion $\E \hookrightarrow \KHausB$ preserves limits, which can be done by 
    replacing $\twoirr$ with $\twoirrsym$ in the proof of~\eqref{thm: preservation limits: irreflexive}.
    
    \eqref{thm: preservation limits: reflexive symmetric}.
    Let $\KHausTB$ denote the full subcategory of $\KHausR$ consisting of those $(X,R)$ where $R$ is reflexive and symmetric.
    By the proof of Proposition~\ref{Prop: KHausR complete}, the inclusion $\KHausTB \allowbreak \hookrightarrow \KHausR$ preserves limits. Thus, it is enough to prove that the inclusion $\E \hookrightarrow \KHausTB$ preserves limits, which can be done by replacing $\tworef$ by $\tworefsym$ in the proof of Theorem~\ref{thm: preservation limits: reflexive}.
\end{proof}

\begin{remark}
    Comparing Theorem~\ref{thm: preservation limits} to Theorem~\ref{thm: preservation limits: reflexive}, in the reflexive case we don't need $\E$ to contain the reflexive singleton because every constant map from $\tworef$ to an object $(X,R)$ with $R$ reflexive is a morphism of $\KHausR$.
    This does not hold in general if $R$ is not reflexive: indeed, if $y$ is an irreflexive point, then the constant map $c_y \colon \twoirr \to X$ with constant value $y$ is not a morphism.
\end{remark}

\begin{theorem}\label{thm: continuous}
    Each of the categories 
    $\StoneC$, $\StoneCT$, $\StoneCKfour$, $\StoneCB$, \allowbreak
    $\StoneCTB$, and $\StoneCSfive$
    lacks some equalizers.
\end{theorem}

\begin{proof}
    The same proof as for $\StoneCSfour$ gives that $\StoneCT$ lacks some equalizers. Let $\E$ denote one of the remaining categories. We adjust the proof of Theorem~\ref{thm: continuous StoneC} by taking the equivalence relations generated by the partial orders of the spaces $X$ and $X'$ shown in  Figures~\ref{fig1} and~\ref{fig2}: 
     
    \begin{figure}[!h]
        \centering
        \begin{tikzpicture}
        	\begin{pgfonlayer}{nodelayer}
        		\node [style=white dot] (0) at (-3, 0.5) {};
        		\node [style=white dot] (1) at (-3, 2) {};
        		\node [style=white dot] (2) at (-2, 2) {};
        		\node [style=white dot] (3) at (-2, 0.5) {};
        		\node [style=white dot] (4) at (-1, 0.5) {};
        		\node [style=white dot] (5) at (-1, 2) {};
        		\node [style=white dot] (6) at (2, 2) {};
        		\node [style=white dot] (7) at (2, 0.5) {};
        		\node [style=white dot] (8) at (-3, 0.5) {};
        		\node [style=small black dot] (9) at (0, 2) {};
        		\node [style=small black dot] (10) at (0.5, 2) {};
        		\node [style=small black dot] (11) at (1, 2) {};
        		\node [style=small black dot] (12) at (0, 0.5) {};
        		\node [style=small black dot] (13) at (0.5, 0.5) {};
        		\node [style=small black dot] (14) at (1, 0.5) {};
        		\node [style=none] (19) at (-3, -0.25) {$x_0$};
        		\node [style=none] (20) at (-2, -0.25) {$x_1$};
        		\node [style=none] (21) at (-1, -0.25) {$x_2$};
        		\node [style=none] (22) at (2, -0.25) {$x_\infty$};
        		\node [style=none] (23) at (-3, 2.75) {$y_0$};
        		\node [style=none] (24) at (-2, 2.75) {$y_1$};
        		\node [style=none] (25) at (-1, 2.75) {$y_2$};
        		\node [style=none] (26) at (2, 2.75) {$y_\infty$};
        	\end{pgfonlayer}
        	\begin{pgfonlayer}{edgelayer}
        		\draw[style=doubleto] (8) to (1);
        		\draw[style=doubleto] (3) to (2);
        		\draw[style=doubleto] (4) to (5);
        		\draw[style=doubleto] (7) to (6);
        	\end{pgfonlayer}
        \end{tikzpicture}
        \caption{The space $(X,R)$.}
        \label{fig1a}
    \end{figure}

    \begin{figure}[!h]
        \centering
        \begin{tikzpicture}
        	\begin{pgfonlayer}{nodelayer}
        		\node [style=white dot] (1) at (-4.5, 2) {};
        		\node [style=white dot] (2) at (-2.75, 2) {};
        		\node [style=white dot] (3) at (-3.25, 0.5) {};
        		\node [style=white dot] (4) at (-1.5, 0.5) {};
        		\node [style=white dot] (5) at (-1, 2) {};
        		\node [style=white dot] (6) at (2, 2) {};
        		\node [style=white dot] (7) at (2, 0.5) {};
        		\node [style=white dot] (8) at (-5, 0.5) {};
        		\node [style=small black dot] (9) at (0, 2) {};
        		\node [style=small black dot] (10) at (0.5, 2) {};
        		\node [style=small black dot] (11) at (1, 2) {};
        		\node [style=small black dot] (12) at (0, 0.5) {};
        		\node [style=small black dot] (13) at (0.5, 0.5) {};
        		\node [style=small black dot] (14) at (1, 0.5) {};
        		\node [style=none] (19) at (-5, -0.25) {$u_0$};
        		\node [style=none] (20) at (-3.25, -0.25) {$u_1$};
        		\node [style=none] (21) at (-1.5, -0.25) {$u_2$};
        		\node [style=none] (22) at (2, -0.25) {$u_\infty$};
        		\node [style=none] (23) at (-5.75, 2.5) {$v_0$};
        		\node [style=none] (24) at (-3.75, 2.5) {$v_1$};
        		\node [style=none] (25) at (-2, 2.5) {$v_2$};
        		\node [style=none] (26) at (2, 2.5) {$v_\infty$};
        		\node [style=white dot] (27) at (-5.5, 2) {};
        		\node [style=white dot] (28) at (-3.75, 2) {};
        		\node [style=white dot] (30) at (-2, 2) {};
        		\node [style=none] (31) at (-4.5, 2.5) {$w_0$};
        		\node [style=none] (32) at (-2.75, 2.5) {$w_1$};
        		\node [style=none] (33) at (-1, 2.5) {$w_2$};
        	\end{pgfonlayer}
        	\begin{pgfonlayer}{edgelayer}
        		\draw[style=doubleto] (8) to (1);
                    \draw[style=doubleto] (27) to (1);
        		\draw[style=doubleto] (3) to (2);
        		\draw[style=doubleto] (4) to (5);
        		\draw[style=doubleto] (30) to (5);
        		\draw[style=doubleto] (7) to (6);
        		\draw[style=doubleto] (8) to (27);
        		\draw[style=doubleto] (3) to (28);
                    \draw[style=doubleto] (2) to (28);
        		\draw[style=doubleto] (4) to (30);
        	\end{pgfonlayer}
        \end{tikzpicture}
        \caption{The space $(X',R')$.}
        \label{fig2a}
    \end{figure}

    It is not difficult to see that $(X,R)$ and $(X',R')$ are objects of $\StoneCSfive$, and hence of $\E$.
    Now define $f,g\colon X \to X'$ as in the proof of Theorem~\ref{thm: continuous StoneC} and observe that the same proof yields that
    the relation on the equalizer of $f$ and $g$ in $\KHausR$ is no longer continuous.
    Thus, by Theorem~\ref{thm: preservation limits}, the equalizer of $f$ and $g$ does not exist in $\E$.
\end{proof}

Putting Theorems~\ref{thm: JT for BAOs} and~\ref{thm: continuous} together   immediately gives:

\begin{corollary}
    Each of $\BAO$, $\T$, $\Kfour$, $\B$, $\TB$, and $\MA$ lacks some coequalizers.
    Consequently, none of them is equivalent to a prevariety.
\end{corollary}

\section{Heyting algebras and frames with lattice morphisms}\label{sec: 5}

In this final section we first show that the category of frames with Heyting morphisms lacks some binary copowers. We then utilize Priestley duality for bounded distributive lattices to show that the category of Heyting algebras with bounded lattice morphisms lacks some coequalizers and its full subcategory consisting of frames lacks some binary copowers. Thus, none of these categories is equivalent to a prevariety.

We recall that a bounded distributive lattice $A$ is a \emph{Heyting algebra} if the set 
\[
\{ x \in A : a \wedge x \le b \}
\]
has a largest element for each $a,b \in A$. A \emph{frame} is a complete Heyting algebra (that is, a complete lattice satisfying the join-infinite distributive law 
\[
a \wedge\bigvee S = \bigvee \{ a\wedge s : s \in S \}
\]
for each $a\in A$ and $S \subseteq A$).

\begin{definition}
Let $\DLat$ denote the category of bounded distributive lattices with bounded lattice morphisms, and let $\Heyt$ denote the category of Heyting algebras with Heyting morphisms.
\end{definition}

We also consider the following categories:

\begin{definition}
    Let $\HeytBL$ denote the full subcategory of $\DLat$ consisting of Heyting algebras, $\FrmBL$ the full subcategory of $\HeytBL$ consisting of frames, and $\FrmHA$ the full subcategory of $\Heyt$ consisting of frames.    
\end{definition}

\begin{remark}
    The category $\HeytBL$ is different from $\Heyt$ in that the morphisms in $\HeytBL$ are bounded lattice morphisms, which in general do not preserve Heyting implication. Similarly, $\FrmHA$ and $\FrmBL$ have the same objects, but differ at the morphism level, and both are different from the category of frames and frame morphisms.
\end{remark}

By \cite[Cor.~IX.5.4]{BD74}, there is a reflector $\mathcal{R} \colon \Heyt \to \BA$, which restricts to a reflector $\mathcal{R} \colon \FrmHA \to \cBA$ by \cite[Thm.~VIII.4.4]{BD74}. Since $\cBA$ is a full subcategory of $\FrmHA$, $\mathcal{R}$ is essentially surjective by \cite[Thm.~I.18.4]{BD74}. Because $\mathcal{R}$ is a left adjoint, it preserves colimits. Thus, Corollary~\ref{cor:no-bin-prod-if-surjects} applies, yielding the following:

\begin{theorem}\label{Thm: FrmHA not cocomplete}
    $\FrmHA$ lacks some binary copowers, and hence is not equivalent to a prevariety.
\end{theorem}

It remains to prove that $\HeytBL$ and $\FrmBL$ are not cocomplete, which we do by using Priestley duality and its restrictions to Heyting algebras and frames.

\begin{definition}
    \hfill
    \begin{enumerate}
    
        \item
        A \emph{Priestley space} is a Stone space $X$ equipped with a partial order $\le$ such that for every $x,y \in X$, if $x\not\le y$, then there is a clopen upset $U$ such that $x\in U$ and $y\notin U$.
        
        \item Let $\Pries$ denote the category of Priestley spaces and continuous order-preserving maps between them.
        
    \end{enumerate}
\end{definition}

It is well known (see, e.g., \cite[Prop.~2.6]{Pri84}) that if $X$ is a Priestley space, then $\le$ is a closed relation on $X$. Therefore, $\Pries$ is a full subcategory of $\KHausR$ because $\Pries$-morphisms are precisely stable maps between Priestley spaces viewed as objects of $\KHausR$.  

\begin{theorem} {\emph{(Priestley duality \cite{Pri70})}} 
    $\DLat$ is dually equivalent to $\Pries$.
\end{theorem}

\begin{definition}
    \hfill
    \begin{enumerate}
        \item An \emph{Esakia space} is a Priestley space in which the downset of every clopen subset is clopen.
        \item A \emph{localic space} \textup{(}\emph{L-space}, for short\textup{)} is a Priestley space in which the closure of any open upset is a clopen upset.
        \item Let $\Esast$ and $\LPries$ denote the full subcategories of $\Pries$ consisting of Esakia spaces and L-spaces, respectively.
    \end{enumerate}
\end{definition}

\begin{remark}
    It is well known that each L-space is an Esakia space. This, for example, can be seen by observing that a Priestley space is an Esakia space if and only if the closure of each open upset is an upset (see, e.g., \cite[Lem.~4.2]{BGJ13}).
\end{remark}

It follows from the results of \cite{Esa74,PS88} that Priestley duality restricts to $\HeytBL$ and $\FrmBL$:

\begin{theorem}\label{thm: duality FrmBL and HeytBL}
Priestley duality restricts to a dual equivalence between $\HeytBL$ and $\Esast$ and to a dual equivalence between $\FrmBL$ and $\LPries$. 
\end{theorem}

\begin{theorem}\label{thm: LPries and Esa not complete}
\hfill\begin{enumerate}
\item\label{thm: LPries and Esa not complete: 1} $\Esast$ is not complete as it lacks some equalizers.
\item\label{thm: LPries and Esa not complete: 2} $\LPries$ is not complete as it lacks some binary powers.
\end{enumerate}
\end{theorem}

\begin{proof}
    \eqref{thm: LPries and Esa not complete: 1}. It is sufficient to observe that the spaces considered in the proof of Theorem~\ref{thm: continuous StoneC} are Esakia spaces, hence the same proof yields that $\Esast$ lacks some equalizers.

    \eqref{thm: LPries and Esa not complete: 2}. 
    We view each ED-space as an L-space where the relation is the identity relation. We claim that $\LPries$ lacks the product of $\beta \N$ with itself. By Theorem~\ref{thm: preservation limits: reflexive}, the inclusion $\LPries \hookrightarrow \KHausR$ preserves limits. Therefore, if the product of $\beta\mathbb N$ with itself existed in $\LPries$, it would be the topological product $\beta\N \times \beta \N$ by the proof of Proposition~\ref{Prop: KHausR complete}, which is a contradiction since $\beta\N \times \beta \N$ is not extremally disconnected. 
\end{proof}

Putting Theorems~\ref{thm: duality FrmBL and HeytBL} and \ref{thm: LPries and Esa not complete} together yields: 

\begin{corollary}\label{Cor: HeytBL and FrmBL not cocomplete}
    \hfill
    \begin{enumerate}
    
        \item\label{Cor: HeytBL and FrmBL not cocomplete : 1}
        $\HeytBL$ is not cocomplete as it lacks some coequalizers. Thus, $\HeytBL$ is not equivalent to a prevariety.
        
        \item\label{Cor: HeytBL and FrmBL not cocomplete : 2}
        $\FrmBL$ is not cocomplete as it lacks some binary copowers. Thus, $\FrmBL$ is not equivalent to a prevariety.
    \end{enumerate}
\end{corollary}

The table below summarizes the main results of this section.

\begin{table}[h!]
    \centering
    \begin{tabular}{|l|l|l|l|}
        \hline
        \textbf{Category} & \textbf{Objects} & \textbf{Morphisms} & \textbf{Location} \\
        \hline
        $\HeytBL$ & Heyting algebras & Bounded lattice morphisms & Cor.~\ref{Cor: HeytBL and FrmBL not cocomplete}\eqref{Cor: HeytBL and FrmBL not cocomplete : 1}\\
        $\FrmBL$ & Frames & Bounded lattice morphisms & Cor.~\ref{Cor: HeytBL and FrmBL not cocomplete}\eqref{Cor: HeytBL and FrmBL not cocomplete : 2} \\
        $\FrmHA$ & Frames & Heyting morphisms & Thm.~\ref{Thm: FrmHA not cocomplete}\\
        \hline
    \end{tabular}
    \vspace{3mm}
    \caption{Categories of Heyting algebras and frames that are not cocomplete.}\label{table 6}
\end{table}

\section*{Acknowledgments}
We would like to thank the referee for useful comments on the structure of the paper, which have improved the readability. 
We are also grateful to all the participants of the \href{https://sites.google.com/view/frametheoryseminar/frame-theory-workshop?authuser=0}{Frame Theory Workshop} (Chapman University, December 2024).
The first author was funded by UK Research and Innovation (UKRI) under the UK government’s Horizon Europe funding guarantee (grant number EP/Y015029/1, Project ``DCPOS'') during his affiliation at the University of Birmingham and by an FSR Incoming Postdoctoral Fellowship during his affiliation at the Université catholique de Louvain.


\begin{thebibliography}{BdRV01}

\bibitem[AHS06]{AHS06}
J.~Ad{\'a}mek, H.~Herrlich, and G.~E. Strecker.
\newblock Abstract and concrete categories: the joy of cats.
\newblock {\em Repr. Theory Appl. Categ.}, (17):1--507, 2006.

\bibitem[BBI16]{BBI16}
G.~Bezhanishvili, N.~Bezhanishvili, and R.~Iemhoff.
\newblock Stable canonical rules.
\newblock {\em J. Symb. Log.}, 81(1):284--315, 2016.

\bibitem[BD74]{BD74}
R.~Balbes and P.~Dwinger.
\newblock {\em Distributive lattices}.
\newblock University of Missouri Press, Columbia, Mo., 1974.

\bibitem[BdRV01]{BRV01}
P.~Blackburn, M.~de~Rijke, and Y.~Venema.
\newblock {\em Modal logic}, volume~53 of {\em Cambridge Tracts in Theoretical Computer Science}.
\newblock Cambridge University Press, Cambridge, 2001.

\bibitem[BGJ13]{BGJ13}
G.~Bezhanishvili, D.~Gabelaia, and M.~Jibladze.
\newblock Funayama's theorem revisited.
\newblock {\em Algebra Universalis}, 70(3):271--286, 2013.

\bibitem[BK24]{BK24}
G.~Bezhanishvili and A.~Kornell.
\newblock The category of topological spaces and open maps does not have products.
\newblock {\em Adv. Math.}, 458:Paper No. 109963, 2024.

\bibitem[BR23]{BR23}
G.~Bezhanishvili and R.~Raviprakash.
\newblock Mc{K}insey-{T}arski algebras: an alternative pointfree approach to topology.
\newblock {\em Topology Appl.}, 339:Paper No. 108689, 2023.

\bibitem[CZ97]{CZ97}
A.~Chagrov and M.~Zakharyaschev.
\newblock {\em Modal logic}.
\newblock Oxford University Press, New York, 1997.

\bibitem[dJ80]{DeJ80}
D.~H.~J. de~Jongh.
\newblock A class of intuitionistic connectives.
\newblock In {\em The {K}leene {S}ymposium ({P}roc. {S}ympos., {U}niv. {W}isconsin, {M}adison, {W}is., 1978)}, volume 101 of {\em Stud. Logic Found. Math.}, pages 103--111. North-Holland, Amsterdam-New York, 1980.

\bibitem[Esa74]{Esa74}
L.~L. Esakia.
\newblock Topological {K}ripke models.
\newblock {\em Soviet Math. Dokl.}, 15:147--151, 1974.

\bibitem[Gai64]{Gai64}
H.~Gaifman.
\newblock Infinite {B}oolean polynomials. {I}.
\newblock {\em Fund. Math.}, 54:229--250, 1964.

\bibitem[GH09]{GH09}
S.~Givant and P.~Halmos.
\newblock {\em Introduction to {B}oolean algebras}.
\newblock Undergraduate Texts in Mathematics. Springer, New York, 2009.

\bibitem[Ghi10]{Ghi10}
S.~Ghilardi.
\newblock Continuity, freeness, and filtrations.
\newblock {\em J. Appl. Non-Classical Logics}, 20(3):193--217, 2010.

\bibitem[GJ60]{GJ60}
L.~Gillman and M.~Jerison.
\newblock {\em Rings of continuous functions}.
\newblock The University Series in Higher Mathematics. D. Van Nostrand Co., Inc., Princeton, N.J.-Toronto-London-New York, 1960.

\bibitem[Hal56]{Hal56}
P.~R. Halmos.
\newblock Algebraic logic. {I}. {M}onadic {B}oolean algebras.
\newblock {\em Compositio Math.}, 12:217--249, 1956.

\bibitem[Hal64]{Hal64}
A.~W. Hales.
\newblock On the non-existence of free complete {B}oolean algebras.
\newblock {\em Fund. Math.}, 54:45--66, 1964.

\bibitem[Joh82]{Joh82}
P.~T. Johnstone.
\newblock {\em Stone spaces}, volume~3 of {\em Cambridge Studies in Advanced Mathematics}.
\newblock Cambridge University Press, Cambridge, 1982.

\bibitem[JT51]{JT51}
B.~J{\'o}nsson and A.~Tarski.
\newblock Boolean algebras with operators. {I}.
\newblock {\em Amer. J. Math.}, 73:891--939, 1951.

\bibitem[MT44]{MT44}
J.~C.~C. McKinsey and A.~Tarski.
\newblock The algebra of topology.
\newblock {\em Ann. of Math.}, 45:141--191, 1944.

\bibitem[Pri70]{Pri70}
H.~A. Priestley.
\newblock Representation of distributive lattices by means of ordered {S}tone spaces.
\newblock {\em Bull. London Math. Soc.}, 2:186--190, 1970.

\bibitem[Pri84]{Pri84}
H.~A. Priestley.
\newblock Ordered sets and duality for distributive lattices.
\newblock In {\em Orders: description and roles ({L}'{A}rbresle, 1982)}, volume~99 of {\em North-Holland Math. Stud.}, pages 39--60. North-Holland, Amsterdam, 1984.

\bibitem[PS88]{PS88}
A.~Pultr and J.~Sichler.
\newblock Frames in {P}riestley's duality.
\newblock {\em Cah. Topol. G\'{e}om. Diff\'{e}r. Cat\'{e}g.}, 29(3):193--202, 1988.

\bibitem[Rie16]{Riehl2016}
E.~Riehl.
\newblock {\em Category theory in context}.
\newblock Mineola, NY: Dover Publications, 2016.

\bibitem[RS70]{RS70}
H.~Rasiowa and R.~Sikorski.
\newblock {\em The mathematics of metamathematics}, volume~41.
\newblock PWN---Polish Scientific Publishers, Warsaw, third edition, 1970.
\newblock Monografie Matematyczne.

\bibitem[SR99]{SR99}
J.~D.~H. Smith and A.~B. Romanowska.
\newblock {\em Post-modern algebra}.
\newblock Pure and Applied Mathematics. John Wiley \& Sons, Inc., New York, 1999.

\end{thebibliography}
\end{document}